\documentclass{amsart}

\usepackage{amsthm, amssymb}
\usepackage[usenames]{color}
\usepackage[pdftex]{graphicx} 

\newtheorem{theorem}{Theorem}[section]

\newtheorem{lemma}[theorem]{Lemma}

\theoremstyle{definition}

\DeclareSymbolFont{bbgreek}{U}{bbold}{m}{n} 

\DeclareMathSymbol{\bbalpha}{\mathbb}{bbgreek}{'13}
\DeclareMathSymbol{\bbbeta}{\mathbb}{bbgreek}{'14}
\DeclareMathSymbol{\bbchi}{\mathbb}{bbgreek}{'37}
\DeclareMathSymbol{\bbdelta}{\mathbb}{bbgreek}{'16}
\DeclareMathSymbol{\bbepsilon}{\mathbb}{bbgreek}{'17}
\DeclareMathSymbol{\bbphi}{\mathbb}{bbgreek}{'36}
\DeclareMathSymbol{\bbgamma}{\mathbb}{bbgreek}{'15}
\DeclareMathSymbol{\bbeta}{\mathbb}{bbgreek}{'21}
\DeclareMathSymbol{\bbiota}{\mathbb}{bbgreek}{'23}
\DeclareMathSymbol{\bbkappa}{\mathbb}{bbgreek}{'24}
\DeclareMathSymbol{\bblambda}{\mathbb}{bbgreek}{'25}
\DeclareMathSymbol{\bbmu}{\mathbb}{bbgreek}{'26}
\DeclareMathSymbol{\bbnu}{\mathbb}{bbgreek}{'27}
\DeclareMathSymbol{\bbpi}{\mathbb}{bbgreek}{'31}
\DeclareMathSymbol{\bbtheta}{\mathbb}{bbgreek}{'22}
\DeclareMathSymbol{\bbrho}{\mathbb}{bbgreek}{'32}
\DeclareMathSymbol{\bbsigma}{\mathbb}{bbgreek}{'33}
\DeclareMathSymbol{\bbvarsigma}{\mathbb}{bbgreek}{'20}
\DeclareMathSymbol{\bbtau}{\mathbb}{bbgreek}{'34}
\DeclareMathSymbol{\bbupsilon}{\mathbb}{bbgreek}{'35}
\DeclareMathSymbol{\bbomega}{\mathbb}{bbgreek}{'177}
\DeclareMathSymbol{\bbxi}{\mathbb}{bbgreek}{'30}
\DeclareMathSymbol{\bbpsi}{\mathbb}{bbgreek}{'40}

\DeclareMathSymbol{\bbzero}{\mathbb}{bbgreek}{'60}
\DeclareMathSymbol{\bbunit}{\mathbb}{bbgreek}{'61}

\DeclareMathSymbol{\bbAlpha}{\mathbb}{bbgreek}{'101}
\DeclareMathSymbol{\bbBeta}{\mathbb}{bbgreek}{'102}
\DeclareMathSymbol{\bbChi}{\mathbb}{bbgreek}{'130}
\DeclareMathSymbol{\bbDelta}{\mathbb}{bbgreek}{'1}
\DeclareMathSymbol{\bbEpsilon}{\mathbb}{bbgreek}{'105}
\DeclareMathSymbol{\bbPhi}{\mathbb}{bbgreek}{'10}
\DeclareMathSymbol{\bbGamma}{\mathbb}{bbgreek}{'0}
\DeclareMathSymbol{\bbEta}{\mathbb}{bbgreek}{'110}
\DeclareMathSymbol{\bbIota}{\mathbb}{bbgreek}{'111}
\DeclareMathSymbol{\bbKappa}{\mathbb}{bbgreek}{'113}
\DeclareMathSymbol{\bbLambda}{\mathbb}{bbgreek}{'3}
\DeclareMathSymbol{\bbMu}{\mathbb}{bbgreek}{'115}
\DeclareMathSymbol{\bbNu}{\mathbb}{bbgreek}{'116}
\DeclareMathSymbol{\bbO}{\mathbb}{bbgreek}{'117}
\DeclareMathSymbol{\bbPi}{\mathbb}{bbgreek}{'5}\DeclareMathSymbol{\bbTheta}{\mathbb}{bbgreek}{'2}
\DeclareMathSymbol{\bbRho}{\mathbb}{bbgreek}{'120}
\DeclareMathSymbol{\bbSigma}{\mathbb}{bbgreek}{'6}
\DeclareMathSymbol{\bbTau}{\mathbb}{bbgreek}{'124}
\DeclareMathSymbol{\bbUpsilon}{\mathbb}{bbgreek}{'131}
\DeclareMathSymbol{\bbOmega}{\mathbb}{bbgreek}{'12}
\DeclareMathSymbol{\bbXi}{\mathbb}{bbgreek}{'4}
\DeclareMathSymbol{\bbPsi}{\mathbb}{bbgreek}{'11}
\DeclareMathSymbol{\bbZeta}{\mathbb}{bbgreek}{'132}

\newcommand{\ignore}[1]{}

\newcommand{\standardfigsized}[3]{
	\begin{figure}
		\centerline{\includegraphics[#3]{figures/#1}}
		\ignore{    \caption{#2 Reference \verb{#1}.}\label{#1}     }
		\caption{#2}\label{#1}
	\end{figure}
}

\definecolor{ComColor}{rgb}{0,.5,.0}

\ignore{\newcommand{\fixme}[1]{}}

\begin{document}

\title{Arc numbers from {G}auss diagrams}
\author{Tobias Hagge}
\begin{abstract}
We characterize planar diagrams which may be divided into $n$ arc embeddings in terms of their chord diagrams, generalizing a result of Taniyama for the case $n = 2$. Two algorithms are provided, one which finds a minimal arc embedding (in quadradic time in the number of crossings), and one which constructs a minimal subdiagram having same arc number as $D$.
\end{abstract}

\maketitle

\section{introduction}
An {\em undecorated chord diagram} $D = (C, H, X)$ consists of an oriented circle $C$, a set $H \subset C$ {\em half-crossings} having finite, even cardinality, and a partition $X$ of $H$ into two element sets, called {\em (full) crossings}. Given another undecorated chord diagram $D' = (C',H',X')$, $D$ and $D'$ are equivalent as undecorated chord diagrams (write $D \sim D'$) if there is an orientation-preserving homeomorphism $\phi: C \to C'$ such that $\{\phi(x)| x \in X\}= X'$. We say $D'$ is a {\em subdiagram} of $D$ if $X' \subset X$ and $H' = \cup X'$.

In an undecorated chord diagram $D=(C,H,X)$, a {\em regular point} $x$ is an element $x \in C \backslash H$. A {\em regular arc} $r$ is an open arc $r \subset C$ such that $\partial r \subset C \backslash H$.  Then $r$ is {\em properly embedded} if it contains no full crossings. An {\em embedded partition} of $D$ is a finite set of regular points, called {\em cut points}, such that the components of the complement are properly embedded arcs. Two embedded partitions of $D$ are equivalent if they differ by an isotopy of $C$ which fixes half-crossings. An arc embedding of $D$ is {\em minimal} if $D$ has no embedded partition with fewer properly embedded arcs. The {\em arc number} of $D$ is the number of arcs in a minimal embedded partition, or, equivalently, the number of cut points.

Arc number first appeared in \cite{Hotz}, where it was shown that every knot has a diagram with arc number two. The result was used to provide a characterization of knot groups in terms of direct products of free groups. The same result was rediscovered in \cite{Ozawa} as part of the construction of a knot invariant. It was used in \cite{A-S-T} to show that every knot contains a diagram containing at most four odd-sided polygons. In \cite{Taniyama}, Taniyama gave a classification of planar diagrams of arc number two in terms of their Gauss diagrams. In (\cite{Taniyama2}) he conjectured the correct classification for arc number three. The present work completes this classification for arbitrary arc number and provides an efficient method for finding minimal embedded partitions for a given diagram.

Two regular arcs $r,r'$ are equivalent as regular arcs if there is a regular arc $r''$ containing both $r$ and $r'$ such that each of $r, r', r'',$ and $r \cap r'$ contain the same set of half-crossings. For example, in a diagram with one crossing there are six equivalence classes of nonempty arcs. Up to arc equivalence, each regular arc is specified by an ordered pair of half crossings $(b(r),f(r))$. The {\em front boundary} $f(r)$ is the hindmost half-crossing in $C\backslash r$, using the preferred orientation; similarly, the {\em back boundary} $b(r)$ is the foremost half-crossing in $C \backslash r$. 

A regular arc $r'$ is a {\em front extension} of a regular arc $r$ if $b(r) = b(r')$ and $r \subset r'$. The extension is {\em proper} if $r \nsim r'$.  A properly embedded regular arc $r$ is an {\em f arc} if no proper front extension of $r$ is properly embedded. Given two inequivalent f-arcs $r$ and $r'$, one may be properly contained in another. In this case $f(r) = f(r')$.

In the sequel, regular arcs and embedded partitions are considered only up to equivalence. For computational purposes each open arc $(h_i, h_{i+1})$ between adjacent half-crossings is identified with a single point, and computations are performed in the quotient topology. For the exposition, however, the language of chord diagrams is retained.

By convention, when points in a finite set $X \subset C$ are expressed with numbered subscripts, the subscripts are defined modulo $\lvert X \rvert$. The cyclic ordering of the subscripts indicates the ordering of the points on $C$, respecting the orientation.

\section{Finding the arc number of an undecorated chord diagram}

\begin{lemma}\label{lmm_mnml_mxml}
Every undecorated chord diagram $D$ has a minimal embedded partition in which at most one component is not an f-arc.
\end{lemma}
\begin{proof}
Let $(p_0, \ldots, p_{n-1})$ be a minimal embedded partition. Let $q_0 = p_0$. Continuing around the circle for $i \in [1,n-2]$, each $q_i$ is given by sliding $p_i$ in the direction of the preferred orientation so that the arc $(q_{i-1},q_i)$ is front maximal. It is not possible to slide $p_i$ past $p_{i+1}$, as this would contradict minimality. The result is a minimal embedded partition in which all components are f-arcs, except perhaps $(q_{n-1},q_{0})$.
\end{proof}

\begin{theorem}\label{thm_cmpttn}
The arc number of an undecorated chord diagram $D=(C,\{h_1, \ldots, h_{2n}\}, X)$ may be computed in $O(n^2)$ time.
\end{theorem}
\begin{proof}
For each regular arc $r = (h_i,h_{i+1})$ in $D$, construct an embedded partition $P = \{q_0, \ldots, q_{n-1}\}$ by placing the first cut point $p_0$ on $r$ and then constructing a sequence of f-arcs, stopping when we reach $p_0$ again. This takes $O(n)$ time. By construction, at least one such $P$ is equivalent to a partition constructed in Lemma~\ref{lmm_mnml_mxml}, and is therefore minimal.
\end{proof}

A {\em star} $S_{t, a}$ for $a \in \mathbb{N}, t \in \mathbb{N}^+$ is an undecorated chord diagram with $1 + (a+1) t$ crossings, these being of the form $c_j = \{h_j, h_{j+ 2 t - 1}\}$, where $j$ ranges over the values in $[0,1+2 (a+1) t]$ having a fixed parity. Up to equivalence $S_{t,a}$ depends only on $t$ and $a$. Some examples of stars are given in Figure~\ref{stars}.

\standardfigsized{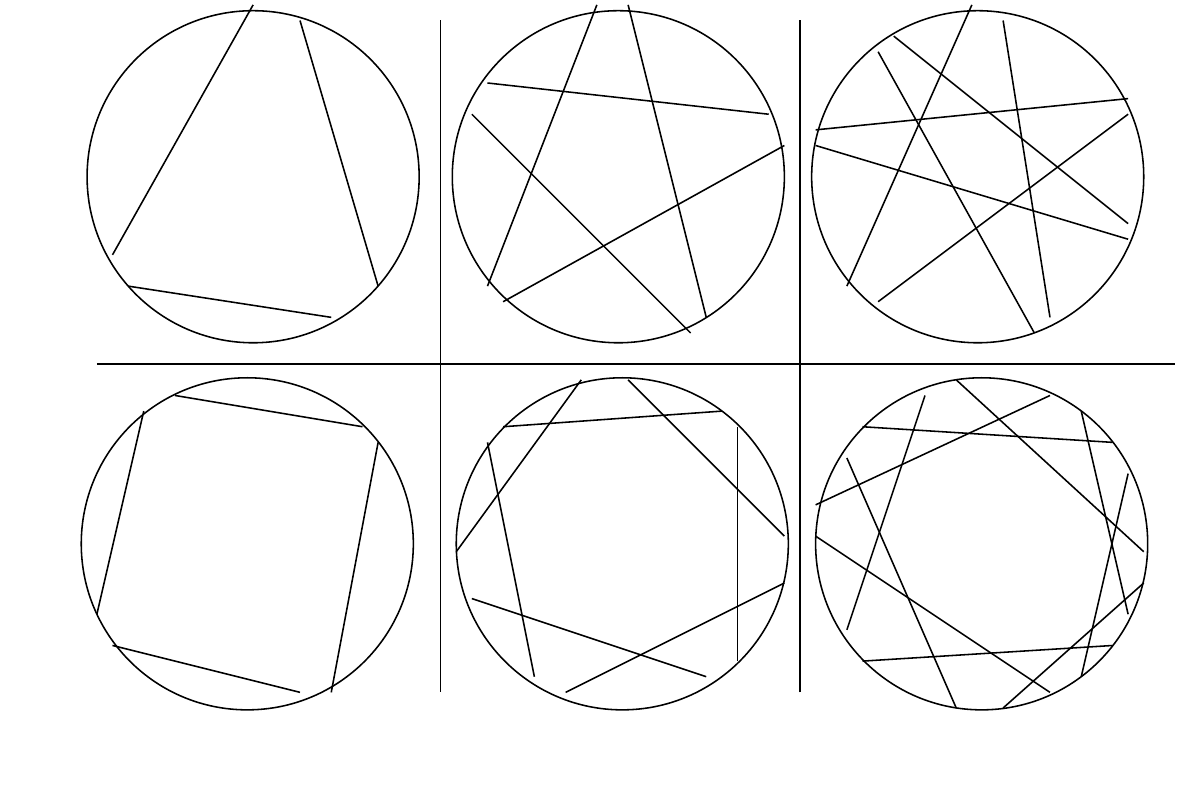}{Some stars. First row: $S_{1,1}, S_{2,1}, S_{3,1}$. Second row: $S_{1,2}, S_{2,2}, S_{3,2}$.}{width=6in}

A {\em star ordering} of $S_{t,a}$ is the ordering $(c_j, c_{j+2t}, \ldots, c_{j+2 (a+1) t t})$ starting with some $j$. Given such, for $n<1 + (a+1)t$ let $S_{t,a}^n$ be the proper subdiagram consisting of the first $n$ crossings in $S_{t,a}^n$. If $t \le t'$ then it is easy to see that $S_{t,a}^{(a+1) t} \cong S_{t',a}^{(a+1) t}$. Thus one may ignore $t$ and write $S_a^n$ without ambiguity. Note that  $S_a^n$ is not a star unless $n \le a+1$, in which case $S_a^{n} = S_{1,n-1}$.

\begin{lemma}\label{lmm_str_nmbr}
The star $S_{t,a}$ has arc number $a+2$.
\end{lemma}
\begin{proof}
The star $S_{t,a}$ is symmetric under a rotation that sends $c_j$ to $c_{j+2}$. Thus one need only construct two embedded partitions to find a minimal one, starting either just before or just after the back half of a crossing. Consider the constructed partitions beginning just before $b(c_0)$ and just after $b(c_{a+1})$. In both cases the first $a+1$ obstructing crossing are the same, forming $S_a^{a+1} \cong S_{1,a-1}$. A single additional arc completes each partition, and both are minimal. \end{proof}

For every f-arc $m$, there is a unique {\em obstructing crossing} $c(m)$ that prevents further front extension. Since $m$ has a preferred orientation, $c_m$ inherits {\em front} (obstructing) and {\em back} half-crossings $f(c(m))$ and $b(c(m))$ respectively from $m$, as well as a preferred properly embedded regular arc $r(c(m)) \sim (b(c(m)),f(c(m))) \subset m$.

If two disjoint f-arcs have the same obstructing crossing, then the diagram has arc number two. In any case, f-arc components $m_1 \ne m_2$ of an embedded partition have disjoint $r(c(m_1))$ and $r(c(m_2))$. More generally, given a pair of disjoint f-arcs $m$ and $m'$, $c(m)$ is never properly contained in $r(c(m'))$.

\begin{theorem}
For $a>0$, an undecorated chord diagram $D$ has arc number $\ge a+2$ iff it contains $S_{t,a}$ as a subdiagram for some $t \in \mathbb{N}^+$.
\end{theorem}

\begin{proof}
Lemma~\ref{lmm_str_nmbr} gives one direction. For the other, the case $a=0$ is obvious. Suppose $a>0$ and $D$ has arc number $a+2$. Since $S_a^{a+1} = S_{1,a-1}$, $S_{t,a}$ contains $S_{1,b}$ for $1 \le b <a$. Thus it suffices to show that $D$ contains $S_{t,a}$ for some $t$.

Construct a minimal embedded partition $P=(p_0, \ldots p_{a+1})$ as in Theorem~\ref{thm_cmpttn}. Since for $i \in [0 \ldots a]$ the arc $(p_i, p_{i+1})$ is front maximal, each $p_i$ with $i>1$ has an associated obstructing crossing $c_{i}$. Since obstructing crossings given by an embedded partition bound disjoint f-arcs, these $c_{i}$ form a copy of $S_{1,a-1} = S_a^{a+1}$.

Now, starting at $p_{a+2}$, continue around the circle, extending embedded arcs to f-arcs by sliding the $p_k$ as was done in Lemma~\ref{lmm_mnml_mxml}. We claim that at some point the generated obstructing crossings give the desired star.

Inductively assume that at the beginning of each stage $k>a+1$, the constructed obstructing crossings $(c_1, \ldots c_{k-1})$ form a star ordered copy of $S_a^{k-1}$. Thus each $f(c_{j})$ lies in $r(c_{j-1-x (a+1)})$ for each $x>0$ such that $j-1-x (a+1) \ge 0$, and similarly each $b(c_{j})$ lies in $r(c_{j-2-x (a+1)})$ for each $x>0$ such that $j-2-x (a+1) \ge 0$.

Construct $c_{k}$ by sliding $p_{k}$ so that $(p_{k-1},p_{k})$ is maximal, and let $c_{k}$ be the associated obstructing crossing. We will show $c_k \ne c_l$ for $l < k$. First, minimality of the resulting partition and front maximality of $p_{k-a}, p_{k-a+1}, \ldots, p_{k}$ imply that $c_{k-a}, c_{k-a+1}, \ldots, c_{k}$ form a star ordered copy of $S_{1,a-1}$, the crossings of which bound disjoint arcs. Second, for each $i < k-a-1$, $f(c_i)$ and $b(c_i)$ each lie in one of $r(c_{k-a-1}), \ldots, r(c_{k-1})$, and so at least one of these lies in one of $r(c_{k-a}), \ldots, r(c_{k-1})$. These two statements imply that $c_{k} \ne c_{l}$ unless possibly $c_{k} = c_{k-a-1}$. In this case, however, $p_{k}$ may be slid all the way to $p_{k-a-1}$, contradicting minimality of the partition. Thus for $l < k$, $c_{j_k} \neq c_{j_l}$.

Consider the locations of $b(c_{k})$ and $f(c_{k})$ relative to the other $c_l$ at stage $k$. There can be no crossing $c_l$ such that $b(c_l)$ lies in $(f(c_{k-1}),b(c_k))$ since otherwise, to avoid proper containment, $f(c_l)$ would lie behind $f(c_k)$ and $c_l$ would be the obstructing crossing at the $k$th stage, not $c_{k}$. Similarly, there can be no $f(c_l) \in (f(c_{k-1}), b(c_k))$, as otherwise either $b(c_{l+1}) \in (b(c_k), f(c_{k-1}))$, in which case $c_{l+1}$ is the obstructing crossing at stage $k$, not $c_{j_k}$, or else $b(c_k) \in (f(c_l),b(c_{l+1}))$, in which case $c_k$ would have been the obstructing crossing at stage $l+1$, rather than $c_{l+1}$. Thus no $c_l$ for $l < k$ has half-crossings in $(f(c_{k-1}), b(c_k))$.

Next, either $f(c_k)$ lies in $r(c_{k-x (a+1)})$ for all $x>0$ such that $k-x (a+1) \ge 0$, or there is a least one $x>0$ such that  $k-x (a+1) \ge 0$ and $f(c_k) \notin r(c_{k-x (a+1)})$. In the first case, the resulting diagram is $S_a^k$. In the second case, since $c_k$ does not contain $c_{k-x(a+1)}$,  $f(c_k)$ lies between$b(c_{k-x (a+1)})$ and $b(c_k)$. In this case the crossings $c_{j_{k-x(a+1)}}, c_{j_{k-x (a+1)+1}}, \ldots, c_{j_k}$ form $S_{x,a}$. Since the number of crossings is finite and the $c_i$ are all distinct, the second case must eventually occur.
\end{proof}


\begin{thebibliography}{99}


\bibitem{A-S-T}{C. Adams, R. Shinjo and K. Tanaka}, {Complementary regions of knot and link diagrams}, arXiv:0812.2558 (2008).

\bibitem{Hotz}{G. Hotz}, {Arkadenfadendarstellung von Knoten und eine neue Darstellung der Knotengruppe} (German), {\it Abh. Math. Sem. Univ. Hamburg}, {\bf 24} (1960), 132-148.

\bibitem{Ozawa}{M. Ozawa}, {Edge number of knots and links}, arXiv:0705.4348 (2007).

\bibitem{Taniyama}{K. Taniyama}, {Circle Immersions that can be divided into two arc embeddings}, {Proc. Amer. Math. Soc}, {\bf 138} (2010), 753-751.

\bibitem{Taniyama2}{K. Taniyama}, {Circle Immersions that can be divided into two arc embeddings}, Knots in Washington XXVIII, Washington D.C., February 28th, 2009.

\end{thebibliography}
\end{document}